\documentclass[reqno]{amsart} \usepackage{amscd}
\usepackage{epsf}
\usepackage{graphicx}
\usepackage{caption}
\usepackage{subcaption}
\allowdisplaybreaks
\newtheorem{theorem}{Theorem}[section]
\newtheorem{proposition}[theorem]{Proposition}

\newtheorem{corollary}[theorem]{Corollary}
\theoremstyle{remark} \newtheorem{remark}[theorem]{Remark}

\usepackage{graphicx}

\begin{document}

\title[]{Finite Decomposition of Minimal surfaces, Maximal surfaces, Time-like Minimal surfaces  and Born-Infeld solitons}

\author{Rukmini Dey*}

\vskip 2mm

\address{I.C.T.S.-T.I.F.R., Bangalore*}

\author{Kohinoor Ghosh**}

\vskip 2mm

\address{I.C.T.S.-T.I.F.R., Bangalore**}

\author{ Sidharth Soundararajan***}

\vskip 2mm

\address{I.I.Sc., Bangalore***}

\begin{abstract}

We show that the  height function of Scherk's second  surface decomposes into a finite sum of scaled and translated versions of itself, using an Euler Ramanujan identity.  A similar result appears in R. Kamien's work on liquid crystals where he shows (using an Euler-Ramanujan identity) that the Scherk's first surface decomposes into a finite sum of scaled and translated versions of itself. We give another  finite decomposition of the height function of the Scherk's first surface in terms of translated  helicoids and scaled and translated Scherk's first surface. We give some more examples, for instance   a (complex) maximal surface and a (complex)  BI soliton. We  then show, using the Weierstrass-Enneper representation of minimal (maximal) surfaces, that one can decompose the height function of a minimal (maximal) surface into finite sums of height functions of surfaces which, upon  change of coordinates, turn out to be  minimal (maximal) surfaces, each minimal (maximal) w.r.t. to its own new coordinates.  We then exhibit a general property of minimal surfaces, maximal surfaces, timelike minimal surfaces and Born-Infeld soliton surfaces that their local height functions $z=Z(x,y)$  split into finite sum of scaled and translated versions of functions of the same form. Upto scaling  these new functions are height functions of the minimal surfaces, maximal surfaces, timelike minimal surfaces and Born-Infeld soliton surfaces respectively. Lastly, we exhibit a foliation of ${\mathbb R}^3$ minus certain lines by shifted helicoids (which appear in one of the Euler-Ramanujan identities).
\end{abstract}

\maketitle

\section{Introduction}

Let ${\mathbb E}^3$ be Euclidean  $3$-space with metric $ds^2 = dx^2 + dy^2 + dz^2$,  ${\mathbb L}^3 $ be Lorentz-Minkowski $3$-space with metric $ds^2 = dx^2 + dy^2 - dz^2$ and ${\mathbb L}^{\prime 3}  $ be the Lorentz-Minkowski $3$-space with $ds^2 = dx^2 - dy^2 +dz^2 $. 

The minimal surface equation for a graphical surface $(x,y, z=z(x,y))$, given by 
$(1+z_x^2) z_{yy} - 2 z_x z_y z_{xy} + (1 + z_y^2) z_{xx} =0$, is the zero mean curvature (ZMC) equation in ${\mathbb E}^3$.  

The maximal surface equation for a graphical surface $(x,y,z=z(x,y)) $ is  given by  $(1-z_x^2) z_{yy} + 2 z_x z_y z_{xy} + (1 - z_y^2) z_{xx} =0$. It is  the zero mean curvature (ZMC) equation in ${\mathbb L}^3$ for a graph over a spacelike plane. 

The Born-Infeld soliton graphical surfaces satisfy the equation $(1-z_y^2) z_{xx}  + 2 z_x z_y z_{xy} - (1+z_x^2) z_{yy} =0$, ~\cite{Wi}. It is the ZMC equation in ${\mathbb L}^{\prime 3}$ for surfaces which are graphs over timelike planes, ~\cite{DeSi}.

The ZMC  equations, when written in terms of the height functions $z=z(x,y)$ are non-linear and we cannot expect the height functions to decompose into finite sums of height functions of zero mean curvature surfaces.  In this article we exhibit that the height functions of various minimal, maximal, timelike minimal and Born-Infeld solitons decompose into sum of translated and scaled versions of minimal, maximal, timelike minimal and BI solitons respectively.

In R. Kamien's work on liquid crystals  ~\cite{Ka} the author decomposes the height function of Scherk's surface of first  kind  in a finite sum of scaled versions of the same type. 
He uses an Euler-Ramanujan identity for this.

He denotes  $ z=h[x,y: \alpha] \equiv -\text{sec}(\frac{1}{2} \alpha) \text{tan}^{-1} \left (  \frac{\text{tanh}[ \frac{1}{2} x \text{sin}(\alpha)]}{\text{tan}[y \text{sin}(\frac{1}{2} \alpha)]} \right ) $.
Using the E-R identity ~\cite{Br}

$$\text{tan}^{-1}[\text{tanh}(a) \text{cot} (b) ] = \sum_{k=-\infty}^{\infty}  \text{tan}^{-1}  ( \frac{a}{b + k \pi}). $$
 Turning it into a finite sum, Kamien shows that 

$$h[x \text{sec} \beta , y ; 2 \beta] = \frac{\text{cos} \tilde{\beta}}{\text{cos} \beta} \sum_{m=0}^{n-1} h[x \text{sec}\tilde{\beta}, y + \frac{m}{n} \pi \text{csc} \tilde{\beta}; 2 \tilde{\beta}]$$ where $\text{sin} \beta = n \text{sin}\tilde{\beta}.$

We use another E-R  identity (which we had used in ~\cite{De} and ~\cite{DeSi}) to give examples of finite decomposition   in the same spirit as Kamien's example.
We  decompose height functions of Scherk's minimal surface of second  kind, i.e. $z = \ln \left( \frac{\text{cos}(x)}{\text{cos}(y)} \right) $ using this E-R identity.   Then we have in theorem (~\ref{scherk})
 $$z(x,y)  = \sum_{m=0}^{n-1} z \left(\frac{x}{n} - c(m), \frac{y}{n} - c(m) \right)$$ where $n$ is any integer and $c(m) = \frac{(2m-n+1)}{2n}\pi.$

We  give a different finite decomposition of the height function of the Scherk's first surface in terms of height functions of translated  helicoids and scaled and translated Scherk's first surface.

We then show using the Weierstrass-Enneper representation of minimal (maximal) surfaces, one can decompose the height function of a minimal (maximal) surface into finite sums of height functions of surfaces which, upon a change of coordinates, turn out to be a minimal (maximal) surface.

We also exhibit a general  property of minimal surfaces in ${\mathbb E}^3$, maximal surfaces in ${\mathbb L}^3 $,  and BI-soliton surfaces in ${\mathbb L}^{\prime 3}$ and timelike minimal surfaces. If locally they are written as a graph of a function, say $(x ,y, z) $ such that  $z = Z(x,y)$,
one can decompose the height function, $$z=Z(x,y)  = \frac{1}{C_n} \sum_{m=1}^{n} Z_{m}(a_mx + b_m , a_my + d_m)$$ where $z=  Z_{m}(x,y)$ is again a surface of the same type, namely scaled and translated minimal surfaces, maximal surfaces, timelike minimal surfaces and Born-Infeld solitons respectively.   In fact $Z_m(x,y) = \frac{1}{c_m} Z(\frac{x-b_m}{a_m}, \frac{y-d_m}{a_m})$ and $C_n = \sum_{m=1}^{n}\frac{1}{c_m}$ is a constant, $c_m$ is a sequence of non-zero real numbers and $z=\alpha_{m} Z_m(x,y)$ is a minimal surface, maximal surface, timelike minimal surface or BI soliton respectively.  We show this using  the Weierstrass-Enneper representation of minimal surfaces, maximal surfaces, timelike minimal surfaces  respectively ~\cite{DHKW}, ~\cite{Ki}, ~\cite{Le}   and the Barbishov-Charnikov representation of  BI-soliton surfaces ~\cite{Wi}.  
Note that our method extends to  cases when the surface can be written as $x=X(y,z)$ or $y = Y(x,z)$. There is nothing special about the $z$ coordinate.

We use the term complex minimal, complex maximal and complex   BI soliton surfaces for solutions of the respective equations  (~\ref{minsurf}), (~\ref{maxsurf}), (~\ref{(BI)}) with $(x,y,z)$ complex.  A  complex maximal surface can be obtained from the complex minimal surface equation by $x \rightarrow ix$ and $y \rightarrow iy$.  A complex Scherk's maximal surface of second kind is given by  $z = \ln \left( \frac{\text{cosh}(x)}{\text{cosh}(y)} \right) $ (where we allow $(x,y,z)$  to be complex.)  We have in  proposition (~\ref{scherkmax})
$$z(x,y)=\sum_{m=0}^{n-1} z \left(\frac{x}{n}+ic(m),\frac{y}{n}+ic(m) \right)$$ where $c(m) = \frac{2m-n+1}{2n}\pi.$

A complex BI-soliton surface equation can be obtained from the complex  minimal surface equation by a Wick rotation $y \rightarrow i y$ and thus $z= \ln \left(  \frac{\text{cosh}(y)}{\text{cos}(x)} \right) $ is a complex  BI-soliton surface. 
We have in proposition (~\ref{scherkBI}) $$ z(x,y) =\sum_{m=0}^{n-1} z\left( \frac{x}{n}- c(m), \frac{y}{n}+ic(m)\right) $$ where $ c(m) = \frac{(2m-n+1)}{2n}\pi.$

 Finally we exhibit a  foliation of ${\mathbb R}^3$ minus certain lines  by shifted helicoids (which appear in one of the Euler-Ramanujan identities).

The results in this paper may have application in the theory of liquid crystals, especially in analysing  screw dislocaltions and topological defects ~\cite{Ka}, ~\cite{KaLu}.

\section{{\bf Finite decomposition of the height function of the Scherk's surface of second kind and the helicoid}}

\subsection{Decomposition of height function of the Scherk's surface of the second kind}

Let us consider the decomposition of Scherk's surface of first kind ~\cite{Ka}.

We show a similar decomposition of Scherk's surface of second kind using an E-R identity used  in ~\cite{De}.

\begin{theorem}\label{scherk}
Let $z(x,y) = \ln \left(\frac{\text{cos}(y)}{\text{cos}(x)}\right)$ be the height function of the Scherk's second minimal surface. Then we have $z(x,y)  = \sum_{m=0}^{n-1} z \left(\frac{x}{n} - c(m), \frac{y}{n} - c(m) \right)$, where $n$ is any integer and $c(m) = \frac{(2m-n+1)}{2n}\pi$.
\end{theorem}

\begin{proof}
Let $z(x,y) = \ln \left(\frac{\text{cos}(y)}{\text{cos}(x)}\right)$.

We have Ramanujan's identity, ~\cite{Br}, Example $(1)$ page $38$, where 
$X$, $A$ are complex, $A$ is not an odd multiple of $\pi /2$:

$\frac{{\rm cos}(X+A)}{{\rm cos} (A)} = \Pi_{k=1}^{\infty} \left( \left (  1-\frac{X}{(k - \frac{1}{2} \pi) -A} \right )  \left(  1 + \frac{X}{(k - \frac{1}{2} \pi) +A} \right) \right) $.

As in ~\cite{De} we take ln on both sides, to get:

\begin{equation*}
\begin{split}
{\rm ln} \left ( \frac{{\rm cos} (X+A)}{{\rm cos} (A)}\right ) 
&=  \sum_{k=1}^{\infty} {\rm ln} \left( \left ( 1-\frac{X}{(k - \frac{1}{2}) \pi - A} \right)  \left(  1 + \frac{X}{(k - \frac{1}{2}) \pi +A} \right) \right )  \\
&=   \sum_{k=1}^{\infty} {\rm ln} \left(  \left( \frac{ (k - \frac{1}{2} \pi ) -(X + A)}{(k - \frac{1}{2}) \pi - A} \right )  \left(   \frac{ (k - \frac{1}{2}) \pi + (X + A) }{(k - \frac{1}{2}) \pi +A} \right) \right )  
\end{split}
\end{equation*}

The Scherk's second surface is given by 
$ z = {\rm ln} \left ( \frac{{\rm cos}(y)}{{\rm cos}(x)}\right ) $.

Let $X+A = y $ and $A=x$ in Ramanujan's identity.

Then, if $x$ is not an odd multiple of $\frac{\pi}{2}$, we have, 

$$
\ln \left(\frac{{\rm cos}(y)}{{\rm cos}(x)} \right) = \sum_{k=1}^{\infty} {\rm ln} \left(  \left( \frac{ y-(k - \frac{1}{2}) \pi )}{x - (k - \frac{1}{2})\pi} \right) \left(  \frac{y + (k - \frac{1}{2}) \pi}{x + (k - \frac{1}{2})\pi} \right ) \right) 
$$
For $x$ not an odd multiple of $\frac{\pi}{2}$, we have (from Euler-Ramanujan's identity) 

\begin{equation*}
    \begin{split}
        z(x,y) &= \sum_{k=0}^\infty \ln \left( \frac{(2y - (2k+1)\pi)(2y+(2k+1)\pi)}{(2x-(2k+1)\pi)(2x+(2k+1)\pi)} \right)
    \end{split}
\end{equation*}

For convenience, let $f(y,x,k) = \frac{2y - (2k+1)\pi}{2x-(2k+1)\pi}$ and let $g(y,x,k) = \frac{2y + (2k+1)\pi}{2x+(2k+1)\pi}$, so that

$$z(x,y) = \sum_{k=0}^\infty \ln(f(y,x,k)g(y,x,k))$$

Rewriting $k = np + m$ for arbitrary fixed n, $0 \leq m \leq n-1$, we have

\begin{equation*}
    \begin{split}
        z(x,y) &= \sum_{p=0}^\infty \sum_{m=0}^{n-1} \ln \left( \frac{(2y - (2(np+m)+1)\pi)(2y+(2(np+m)+1)\pi)}{(2x-(2(np+m)+1)\pi)(2x+(2(np+m)+1)\pi)} \right)\\
        & = \sum_{p=0}^\infty \sum_{m=0}^{n-1} \ln \left( \frac{(2y - 2m\pi - (2np+1)\pi)(2y+ 2m\pi + (2np+1)\pi)}{(2x- 2m\pi-(2np+1)\pi)(2x+2m\pi + (2np+1)\pi)} \right)\\
        & = \sum_{p=0}^\infty \sum_{m=0}^{n-1} \ln \left( \frac{(2y - 2m\pi +(n-1)\pi - n(2p+1)\pi)(2y+ 2m\pi -(n-1)\pi + n(2p+1)\pi)}{(2x- 2m\pi+(n-1)\pi -n(2p+1)\pi)(2x+2m\pi -(n-1)\pi + n(2p+1)\pi)} \right)\\
    \end{split}
\end{equation*}

Let $c(m) = \frac{(2m-n+1)}{2n}\pi$ . Then

\begin{equation*}
    \begin{split}
        z(x,y) & = \sum_{p=0}^\infty \sum_{m=0}^{n-1} \ln \left( \frac{(\frac{2y}{n} - 2c(m) - (2p+1)\pi)(\frac{2y}{n}+ 2c(m) + (2p+1)\pi)}{(\frac{2x}{n}- 2c(m) -(2p+1)\pi)(\frac{2x}{n}+2c(m) + (2p+1)\pi)} \right)\\ 
        & = \sum_{p=0}^\infty \sum_{m=0}^{n-1} \ln \left( f \left(\frac{y}{n} - c(m), \frac{x}{n} - c(m), p \right ) g \left( \frac{y}{n} +c(m), \frac{x}{n} + c(m), p \right) \right)
    \end{split}
\end{equation*}

Now, $c(n-1-m) = \frac{2(n-1-m) -n +1}{2n}\pi = \frac{2n - 2 -2m -n +1}{2n}\pi = \frac{n-1-2m}{2n}\pi = -c(m)$

On expanding the finite summation, we get a sum of natural log of  $f$ and that of $g$. We can reverse the order of finite summation for the part involving $g$  and rewrite the same expression as follows.

\begin{equation*}
    \begin{split}
        z(x,y) &= \sum_{p=0}^\infty \sum_{m=0}^{n-1} \ln \left( f \left(\frac{y}{n} - c(m), \frac{x}{n} - c(m), p \right ) g \left( \frac{y}{n} +c(n-1-m), \frac{x}{n} + c(n-1-m), p \right) \right)\\
        & = \sum_{p=0}^\infty \sum_{m=0}^{n-1} \ln \left( f \left(\frac{y}{n} - c(m), \frac{x}{n} - c(m), p \right ) g \left( \frac{y}{n} -c(m), \frac{x}{n} -c(m), p \right) \right)\\
         & = \sum_{m=0}^{n-1} z \left(\frac{x}{n} - c(m), \frac{y}{n} - c(m) \right)
    \end{split}
\end{equation*}
interchanging the order of the infinite and the finite sum.

\end{proof}

\subsection{{\bf Finite Decomposition of the height function of the  Scherk's first surface }}

In this section we give a finite decomposition of the height function of the Scherk's first surface in terms of translated  helicoids and scaled and translated Scherk's first surface. It is a different decomposition from ~\cite{Ka}.

\begin{theorem}
\begin{eqnarray*}
    tan^{-1} tanh (y) cot (x) & = \sum_{m=1}^{n-1} tan^{-1} (tanh (y/n)cot ((x + m \pi)/n)) - tan^{-1} \left(\frac{y/n}{(x + m \pi)/n} \right) \nonumber \\
		&  +  tan^{-1}( tanh (y/n)cot (x/n ))  - \sum_{m=1}^{n-1} tan^{-1} \left( \frac{y/n}{(x+ m \pi)/n - \pi} \right)  \\& + \sum_{m=1}^{n-1} tan^{-1} \left( \frac{y}{x + m \pi } \right) +  \sum_{m=1}^{n-1} tan^{-1} \left( \frac{y}{x - m \pi } \right)
\end{eqnarray*}
\end{theorem}

\begin{proof}
Using an Euler-Ramanujan identity, used in this context  in ~\cite{De}, we have $$ \text{tan}^{-1}(tanh (y) cot (x) ) =   tan^{-1} \left(\frac{y}{x} \right) + \sum_{k=1}^{\infty} tan^{-1} (\frac{y}{x + k \pi}) + \sum_{k=1}^{\infty} tan^{-1} (\frac{y}{x - k \pi}) $$

Fix any $n \in {\mathbb N}$, a natural number. Then putting $k = np + m$, $p = 1,..., \infty$ and $m =0,..., n-1$ and the additional case, $p=0, m=1,...,n-1$, we get 

\begin{equation}\label{heli}
    \begin{split}
		 \text{tan}^{-1}(tanh (y) cot (x) )  = &   \sum_{p=1}^\infty \sum_{m=0}^{n-1} \tan^{-1} \left( \frac{y}{x + (np +m) \pi} \right) + \sum_{p=1}^\infty \sum_{m=0}^{n-1} \tan^{-1}\left( \frac{y}{x - (np +m) \pi} \right) \\ & + \sum_{m=1}^{n-1}  tan^{-1} \left( \frac{y}{x + m \pi} \right) +  \sum_{m=1}^{n-1} \tan^{-1} \left( \frac{y}{x  - m \pi} \right) + tan^{-1} \left( \frac{y}{x} \right)
	\end{split}
\end{equation}

Now,  the first two terms are

\begin{align}
		&\sum_{p=1}^\infty \sum_{m=0}^{n-1} \tan^{-1} \left( \frac{y}{x + (np +m) \pi} \right)
		+ \sum_{p=1}^\infty \sum_{m=0}^{n-1} \tan^{-1} \left( \frac{y}{x - (np +m) \pi} \right) \nonumber \\
		 =  & \sum_{p=1}^\infty \sum_{m=0}^{n-1} tan^{-1} \left( \frac{y/n}{(x+m\pi)/n + p \pi} \right) +  \sum_{p=1}^\infty \sum_{m=0}^{n-1} tan^{-1} \left( \frac{y/n}{(x/n - (p \pi + m \pi/n)} \right)\nonumber \\
		&  \text{Putting \;  $m^{\prime} = n-1 -m $ \; in \; the \;  last \; sum \; we \; get}\nonumber \\
		 = & \sum_{p=1}^\infty \sum_{m=0}^{n-1} tan^{-1} \left( \frac{y/n}{(x+m \pi)/n + p \pi} \right) +  \sum_{p=1}^\infty \sum_{m^{\prime}=0}^{n-1} tan^{-1} \left( \frac{y/n}{(x/n - (p \pi + (n-1 -m^{\prime}) \pi/n )} \right) \nonumber \\
		= & \sum_{p=1}^\infty \sum_{m=0}^{n-1} tan^{-1} \left( \frac{y/n}{(x+m \pi)/n + p \pi} \right) +  \sum_{p=1}^\infty \sum_{m=0}^{n-1} tan^{-1} \left( \frac{y/n}{(x+(m +1)\pi) /n - (p+1)\pi } \right) \nonumber \\
		= & \sum_{p=1}^\infty \sum_{m=0}^{n-1} tan^{-1} \left( \frac{y/n}{(x+m\pi)/n + p \pi} \right) +  \sum_{p=2}^\infty \sum_{m=1}^{n} tan^{-1} \left( \frac{y/n}{(x+m\pi) /n - p \pi } \right) \nonumber \\
		= & \sum_{p=1}^\infty \sum_{m=1}^{n-1} tan^{-1} \left( \frac{y/n}{(x+m\pi)/n + p \pi} \right) +  \sum_{p=1}^\infty \sum_{m=1}^{n-1} tan^{-1} \left( \frac{y/n}{(x+m\pi) /n - p \pi } \right) \nonumber \\
		& - \sum_{m=1}^{n-1} tan^{-1} \left( \frac{y/n}{(x+ m \pi)/n - \pi} \right)  +  \sum_{p=1}^{\infty} tan^{-1} \left( \frac{y/n}{(x  /n + p \pi)  } \right)  \nonumber \\ & + \sum_{p = 2}^ {\infty} tan^{-1} \left( \frac{y/n}{(x + n \pi)/n -p \pi} \right) \nonumber \\
		 =  & \sum_{p=1}^\infty \sum_{m=1}^{n-1} tan^{-1} \left( \frac{y/n}{(x+m\pi)/n + p \pi} \right) +  \sum_{p=1}^\infty \sum_{m=1}^{n-1} tan^{-1} \left( \frac{y/n}{(x+m\pi) /n - p \pi } \right) \nonumber \\ &+  \sum_{p=1}^{\infty} tan^{-1} \left( \frac{y/n}{(x  /n + p \pi)  } \right) + \sum_{p = 2}^ {\infty} tan^{-1} \left( \frac{y/n}{(x/n  -(p-1) \pi)} \right) \nonumber \\ & - \sum_{m=1}^{n-1} tan^{-1} \left( \frac{y/n}{(x+ m \pi)/n - \pi} \right)  \nonumber \\
		 = & \sum_{p=1}^\infty \sum_{m=1}^{n-1} tan^{-1} \left( \frac{y/n}{(x+m\pi)/n + p \pi} \right) +  \sum_{p=1}^\infty \sum_{m=1}^{n-1} tan^{-1} \left( \frac{y/n}{(x+m\pi) /n - p \pi } \right) \nonumber \\ 
		&  +  \sum_{p=1}^{\infty} tan^{-1} \left( \frac{y/n}{(x  /n + p \pi)  } \right)  +  \sum_{p = 1}^ {\infty} tan^{-1} \left( \frac{y/n}{(x/n -p \pi)}  \right)\nonumber \\& - \sum_{m=1}^{n-1} tan^{-1} \left( \frac{y/n}{(x+ m \pi)/n - \pi} \right)\nonumber \\ 
		= & \sum_{m=1}^{n-1}  \left(\sum_{p=1}^\infty tan^{-1} \left(\frac{y/n}{(x+m\pi)/n + p \pi} \right) +  \sum_{p=1}^{\infty}   tan^{-1} \left(\frac{y/n}{(x+m\pi) /n - p \pi } \right) \right) \nonumber \\ & - \sum_{m=1}^{n-1} tan^{-1} \left( \frac{y/n}{(x+ m \pi)/n - \pi} \right)\nonumber \\ 
		&  +  \sum_{p=1}^{\infty}  tan^{-1} \left( \frac{y/n}{(x  /n + p \pi)  } \right)  +  \sum_{p=1}^{\infty}   tan^{-1} \left(\frac{y/n}{(x/n -p \pi)} \right) \nonumber \\
		= & \sum_{m=1}^{n-1} tan^{-1} (tanh (y/n)cot ((x + m \pi)/n)) - tan^{-1} \left(\frac{y/n}{(x + m \pi)/n} \right) \nonumber \\
		&  +  tan^{-1}( tanh (y/n)cot (x/n )) -  tan^{-1}(y/x)  - \sum_{m=1}^{n-1} tan^{-1} \left( \frac{y/n}{(x+ m \pi)/n - \pi} \right) \label{heli1}
\end{align}

where  we have used the  Euler- Ramanujan identity, 
$$ \sum_{p=1}^{\infty} tan^{-1} \frac{a}{(b  + p \pi)  }  +  \sum_{p = 1}^ {\infty} tan^{-1}\frac{a}{(b -p \pi)}  + tan^{-1}(a/b)= tan^{-1}(tanh(a) cot(b)) $$

Thus we we get from  equations (\ref{heli}) and (\ref{heli1}) the following equation: 

\begin{equation}
	\begin{split}
		tan^{-1} tanh (y) cot (x)  & = \sum_{m=1}^{n-1} tan^{-1} (tanh (y/n)cot ((x + m \pi)/n)) - tan^{-1} \left(\frac{y/n}{(x + m \pi)/n} \right) \nonumber \\
		&  +  tan^{-1}( tanh (y/n)cot (x/n ))  - \sum_{m=1}^{n-1} tan^{-1} \left( \frac{y/n}{(x+ m \pi)/n - \pi} \right) \\ & + \sum_{m=1}^{n-1} tan^{-1} \left( \frac{y}{x + m \pi } \right)  +  \sum_{m=1}^{n-1} tan^{-1} \left( \frac{y}{x - m \pi } \right)
		\end{split}
 \end{equation}
 \end{proof}

\section{{\bf Finite Decomposition using Weierstrass-Enneper Representation }}

\subsection{Minimal surfaces:}

Let $\Omega$ be a domain  in ${\mathbb C}$ such the following holds. 

Let $\zeta= \zeta_1 + i \zeta_2$ be in $\Omega$. Let ${\bf X}(\zeta_1, \zeta_2) = \left( x(\zeta_1, \zeta_2), y(\zeta_1, \zeta_2), z(\zeta_1, \zeta_2) \right)$ be the Weierstrass-Enneper representation of a minimal surface given by the data $f,g$ as follows:

\begin{equation}
	\begin{split}
		x(\zeta_1, \zeta_2) &= x_0 + \text{Re} \int_{\zeta_0}^{\zeta} ( 1-g^2)(w) f(w)\\
		y(\zeta_1, \zeta_2) &= y_0 + \text{Re} \int_{\zeta_0}^{\zeta} i( 1+g^2)(w) f(w)\\
		z(\zeta_1, \zeta_2) &= z_0 + \text{Re} \int_{\zeta_0}^{\zeta} 2g(w) f(w)
	\end{split}
\end{equation}

Here $f$ is holomorphic and $g$ is meromorphic such that $fg^2$ is holomorphic.

In the neighbourhood of an non-umbilical point, one can perform a change of coodinates such that  the W-E representation take the following form on a domain which we rename as $\Omega$ as well.

 \begin{equation}\label{W-E-R}
 \begin{split}
x(\zeta_1, \zeta_2) = x_0 + \text{Re} \int_{\zeta_0}^{\zeta} ( 1-w^2) R(w) dw \\
y(\zeta_1, \zeta_2) = y_0 + \text{Re} \int_{\zeta_0}^{\zeta} i( 1+w^2) R(w) dw \\
z(\zeta_1, \zeta_2) = z_0 + \text{Re} \int_{\zeta_0}^{\zeta} 2w  R(w) dw 
\end{split}
\end{equation}
where $R=R(w)$ is non-vanishing and is holomorphic  on $\Omega$.

\begin{theorem}
Suppose a minimal surface can be written as $z =z(x,y)$ for $(x,y)$  coordinates on $U$ an open subset of ${\mathbb C}$ such that $(x,y,z)$ is a non-umbilical point on the minimal surface.  Then there exists a neighbourhood $U^{\prime} \subset U$  and $\Phi_i$ invertible maps such that $z (x,y)= z_1 (x,y) + z_2(x,y) + ...+ z_n(x,y)$ where $z_i = z_i( \Phi_i(x_i, y_i))$ are height functions of  minimal surfaces w.r.t. $(x_i, y_i)$, which  are  coordinates in a neighbourhood $W$ of ${\mathbb C}$. This decomposition is not unique. 
\end{theorem}

\begin{proof}
Since $(x,y, z(x,y))$ is non-umbilical, there is a non-vanishing meromorphic function  $R$ which is holomorphic in a neighbourhood $U$  such that (~\ref{W-E-R}) holds in a neighbourhood $U$.
Then  there is a parameter $\zeta  \in \Omega \subset {\mathbb C}$, $\Omega$ open such that $(x,y, z) = (x(\zeta, \bar{\zeta} ) , y(\zeta, \bar{\zeta}), z(\zeta, \bar{\zeta}) )$. 

Let us decompose $R = R_1 + R_2 +....+ R_n$ such that $R_i$ are meromorphic and non-vanishing and holomorphic in the domain $U$ for $i=1,...,n$ such that integrals in equation (~\ref{R}) are finite. Let  $(x_i, y_i, z_i)$ are defined as 

\begin{equation}\label{R}
\begin{split}
 x_i(\zeta_1, \zeta_2) = x_0 + \text{Re} \int_{\zeta_0}^{\zeta} ( 1-w^2) R_i(w) dw \\
y_i(\zeta_1, \zeta_2) = y_0 + \text{Re} \int_{\zeta_0}^{\zeta} i( 1+w^2) R_i(w) dw \\
z_i(\zeta_1, \zeta_2) = z_0 + \text{Re} \int_{\zeta_0}^{\zeta} 2w  R_i(w) dw.
\end{split}
\end{equation}
Then there exists $\Phi_i$ for each $i$, the function $z_i$ can be written as  $z_i = z_i(\Phi_i(x_i, y_i))$.
$\Phi_i$ is defined as follows.

 We know $R_i \neq 0$, $i=1,...,n$. Let $ \eta_i  =  x_i - \sqrt{-1} y_i $ given by (~\ref{R}).   One can show that $\frac{\partial}{\partial \zeta}  \eta_i = R_i \neq 0$.  Then define  $\bar{\zeta}  = f_i(\eta_i)$ where  $f_i$ an invertible map, by inverse function theorem. Then $\zeta = \overline{f_i(\eta_i)}$. Then,  it is easy to see that there exists an invertible map $\Phi_i$ such that $(\zeta, \bar{\zeta}) = \Phi_i (x_i, y_i)$ and thus $z_i = z_i(\zeta, \bar{\zeta}) = z_i ( \Phi_i(x_i, y_i))$.  
$z_i$ is a minimal surface w.r.t.  $(x_i, y_i)$ and (~\ref{R}) is its W-E representation.

\end{proof}

\subsection{Maximal surfaces:}
There is an analogue of the Weierstrass representation for maximal surfaces in ${\mathbb L}^3$ in terms of a non-vanishing meromorphic function $M$ which is holomorphic in a neighbourhood where the Gauss map is $1-1$, see for instance ~\cite{DeSi}. Thus one can show analogously the following theorem:

\begin{theorem}
Suppose a maximal  surface in ${\mathbb L}^3$ can be written as $z =z(x,y)$ for $(x,y)$  coordinates on $U$ an open subset of ${\mathbb C}$ such that the Gauss map is $1-1$.  Then there exists a neighbourhood $U^{\prime} \subset U$  and $\Phi_i$ invertible maps such that $z (x,y)= z_1 (x,y) + z_2(x,y) + ...+ z_n(x,y)$ where $z_i = z_i( \Phi_i(x_i, y_i))$ are height functions of  maximal surfaces w.r.t. $(x_i, y_i)$, which  are  coordinates in a neighbourhood $W$ of ${\mathbb C}$. This decomposition is not unique. 
\end{theorem}

\section{{\bf Decompositions of height functions in general}}

Given $z=f(x,y)$ be a surface given in a local graphical form. Then one can always decompose the height function  into finite sum of scaled versions of itself, namely, $z= \frac{1}{C_n} \sum_{m=1}^n c_m f(x,y) = \frac{1}{C_n} \sum_{m=1}^n  f_m(x,y)$ where $C_n = \sum_{m=1}^n \frac{1}{c_m}$ and $f_m = c_m f$. We will use this fact. 

{\bf Minimal surfaces:}

Let $\Omega$ be a domain  in ${\mathbb C}$. 

Let $\zeta= \zeta_1 + i \zeta_2$ be in $\Omega$. Let ${\bf X}(\zeta_1, \zeta_2) = \left( x(\zeta_1, \zeta_2), y(\zeta_1, \zeta_2), z(\zeta_1, \zeta_2) \right)$ be the Weierstrass-Enneper representation of a minimal surface given by the data $f,g$ as follows:

\begin{equation}
	\begin{split}
		x(\zeta_1, \zeta_2) &= x_0 + \text{Re} \int_{\zeta_0}^{\zeta} ( 1-g^2)(w) f(w) \\
		y(\zeta_1, \zeta_2) &= y_0 + \text{Re} \int_{\zeta_0}^{\zeta} i( 1+g^2)(w) f(w) \\
		z(\zeta_1, \zeta_2) &= z_0 + \text{Re} \int_{\zeta_0}^{\zeta} 2g(w) f(w)
	\end{split}
\end{equation}

Here $f$ is holomorphic and $g$ is meromorphic such that $fg^2$ is holomorphic.

Suppose in this domain $\Omega$,  there exists an invertible map $\Phi$ such that $(\zeta_1, \zeta_2) = \Phi (x,y)$.

Then $z=z \circ \Phi (x,y)$ is height function  of the minimal surface, i.e it satisfies the equation

\begin{equation}\label{minsurf}
(1+ z_x^2) z_{yy}- 2 z_x z_y z_{xy} + (1 + z_y^2) z_{xx} =0
\end{equation}

We name it $z = Z(x,y)$ where $Z(x,y) = z \circ \Phi (x,y)$.

\begin{proposition}\label{decompmin}
Let $\Omega \subset {\mathbb C}$ be a domain and $f,g$ be a Weierstrass-Enneper  data on $\Omega$ for a minimal surface in $ {\mathbb R}^3$ and $\zeta \in \Omega$ be the parameter as above such that $(\zeta_1, \zeta_2) = \Phi (x,y)$ with  $\Phi$ invertible.   Then in  the domain $\Omega$ there exists an invertible map $\Phi_m$  such that $(\zeta_1, \zeta_2) = \Phi_{m}  (a_m x+b_m, a_m y +d_m)$, where  $a_m \neq 0$.   Let $c_m$ be a sequence  of nonzero real numbers.  Then for $n \in \{2,3,...\}$  $z=Z(x,y) = \frac{1}{C_n} \sum_{m=1}^{n } Z_{m} (a_m x+b_m, a_m y+d_m)$ where $Z_m(x,y) = \frac{1}{c_m} Z(\frac{x-b_m}{a_m}, \frac{y-d_m}{a_m})$ and $z=\alpha_m  Z_{m}(x,y)$ is a minimal surface  which is a scaled and translated  version of $z=Z(x,y)$ with $\alpha_m = \frac{c_m}{  a_m }$  and $ C_n = \sum_{m=1}^{n}( \frac{1}{c_m}) $ is a constant. 
\end{proposition}

\begin{proof}

  Recall that  solving $z$ in terms of $(x,y)$ from the W-E data, we have $z=z \circ \Phi (x,y)$ is height function  of the minimal surface where $(\zeta_1, \zeta_2 ) = \Phi (x,y)$.  
We name it $z = Z(x,y)$. Existence of $\Phi$ implies  there exists an invertible map $\Phi_m $ such that $(\zeta_1, \zeta_2) = \Phi_m (a_m x + b_m, a_m y + d_m)$, since $a_m \neq 0$. In fact, $\Phi_m(x,y) = \Phi(\frac{x-b_m}{a_m}, \frac{y-d_m}{a_m})$. 

Let $m \leq  n $ be  integers with $m \geq 2$ . Let $\zeta=  \zeta_1 + i \zeta_2$ and $a_m \neq 0$.

\begin{equation*}
	\begin{split}
		\frac{x(\zeta_1, \zeta_2)}{a_m} &= \frac{x_0}{a_m} + \text{Re} \int_{\zeta_0}^{\zeta} ( 1-g^2)(w) \frac{f(w)}{a_m}\\
		\frac{y(\zeta_1, \zeta_2)}{a_m} &= \frac{y_0}{a_m} + \text{Re} \int_{\zeta_0}^{\zeta} i( 1+g^2)(w) \frac{f(w)}{a_m}\\
		\frac{z(\zeta_1, \zeta_2)}{a_m} &= \frac{z_0}{a_m} + \text{Re} \int_{\zeta_0}^{\zeta} 2g(w) \frac{f(w)}{a_m}
	\end{split}
\end{equation*}

Here $\frac{f}{a_m}$ is holomorphic and $g$ is meromorphic such that $\frac{f}{a_m}g^2$ is holomorphic.
Thus ${\bf X_m} = (\frac{x}{a_m}, \frac{y}{a_m}, \frac{z}{a_m})$ is a minimal surface with the W-E data $(\frac{f}{a_m}, g)$ and the same W-E parameter $ \zeta= \zeta_1 + i \zeta_2 $  as that of ${\bf X}$.   

  Let $z_m (\zeta_1, \zeta_2) = \frac{1}{c_m}z( \zeta_1, \zeta_2 )$.

 Then, $z_m(\zeta_1, \zeta_2) = z_m  \circ \Phi_m (a_m x + b_m, a_m y + d_m)  =     Z_m ( a_m x + b_m, a_m y + d_m) $   where $Z_m = z_m \circ \Phi_m$.

Then $\sum_{m=1}^{n } z_m (\zeta_1, \zeta_2)  =
  \sum_{m=1}^{n}  \frac{1}{c_m} z(\zeta_1 , \zeta_2 ) = z(\zeta_1, \zeta_2) (\sum_{m=1}^{n} \frac{1}{c_m} ) =  C_n z(\zeta_1, \zeta_2 ) $ where $C_n = \sum_{m=1}^{n} \frac{1}{c_m} $.
  
  Thus we have 
  \begin{equation}\label{minsum}
  z(\zeta_1, \zeta_2) = \frac{1}{C_n} \sum_{ m= 1}^{n} z_m(\zeta_1, \zeta_2 )
  \end{equation}

  Then $z \circ \Phi (x,y) = \frac{1}{C_n} \sum_{m=1}^{n} z_m \circ \Phi_m ( a_m x + b_m, a_m y + d_m)$.
  
  Thus we have $ Z(x,y) = \frac{1}{C_n} \sum_{m=1}^{n} Z_m( a_m x + b_m, a_m y + d_m )$   where $Z_m(x,y) = \frac{1}{c_m} Z(\frac{x-b_m}{a_m}, \frac{y-d_m}{a_m})$ and $z=\alpha_m  Z_{m}(x,y)$ is a minimal surface  which is a scaled and translated  version of $z=Z(x,y)$ with $\alpha_m = \frac{c_m}{  a_m }$  and $ C_n = \sum_{m=1}^{n}( \frac{1}{c_m}) $ is a constant.

\end{proof}

\subsection{{\bf Decomposition of a family}}

Let $z=Z(x,y)$ be the local  height function of a minimal (maximal) surface. Let us consider it in the Weierstrass-Enneper representation with parameter $\zeta$. In other words, ${\bf X}=(x,y,z)$ is given by  $x=x(\zeta, \bar{\zeta})$, $y = y(\zeta, \bar{\zeta})$ and $ z=z(\zeta, \bar{\zeta})$ for some harmonic functions of $(\zeta_1, \zeta_2)$ where $x(\zeta, \bar{\zeta}), y(\zeta, \bar{\zeta}), z(\zeta, \bar{\zeta})$  are expressed in the isothermal coordinates $\zeta_1, \zeta_2$, with $\zeta = \zeta_1 + i \zeta_2$. Let ${\bf X^c}(\zeta, \bar{\zeta})  = (x^c(\zeta, \bar{\zeta}), y^c(\zeta, \bar{\zeta}), z^c(\zeta, \bar{\zeta}))$ where 
$x,y, z$ are harmonic conjugates of $x^c, y^c, z^c$, w.r.t $(\zeta_1, \zeta_2)$. 

For each $\theta \in [0, \pi/2]$, let ${\bf X}^{\theta} = \cos(\theta) {\bf X} + \sin(\theta) {\bf X^c} $ which is again a minimal (maximal) surface, ~\cite{Do}. 

Let the domain $\Omega$ be such that for each $\theta$ there exists an invertible transformation $\Phi_{\theta}$ such that $(\zeta_1, \zeta_2) = \Phi_{\theta}(x,y)$. Then there exist an invertible $\Phi^{\theta}_m$ such that
$(\zeta_1, \zeta_2) = \Phi^{\theta}_m( a_m x_{\theta} + b_m, a_m y_{\theta} + d_m)$, where $a_m \neq 0$.  

We have a corollary to proposition (~\ref{decompmin}).

\begin{corollary}\label{decompfamily}
Let $\Omega$ be a domain such that for $\zeta = (\zeta_1, \zeta_2) \in \Omega$  there exists an invertible $\Phi_m^{\theta}$ such that $(\zeta_1, \zeta_2) = \Phi_{m}^{\theta} (a_m x_{\theta} + b_m, a_m y_{\theta} + d_m )$, $a_m \neq 0$. 
Let $z=Z_{\theta}(x_{\theta}, y_{\theta}) $ be the 
height function of the family ${\bf X}^{\theta}$. Then $Z_{\theta}(x_{\theta}, y_{\theta})$ admits a decomposition
 $Z_{\theta}(x_{\theta}, y_{\theta}) = \frac{1}{C_n}  \sum_{m=1}^n Z_{\theta m}(a_m x_{\theta} + b_m, a_m y_{\theta} + d_m)$ where   $ z= \alpha_m Z_{{\theta} m}(x,y) $ are also height functions of minimal  surfaces with $\alpha_m = \frac{c_m}{ a_m}$ and $C_n$ is as above.
\end{corollary}

\begin{proof}
One can apply proposition (~\ref{decompmin}).
\end{proof}

\subsection{{\bf Decomposition of Height functions of Maximal surfaces:}}

Let us consider ${\mathbb L}^3$ be ${\mathbb R}^3$ with the Lorentz Minkowski metric $ds^2 = dx^2 + dy^2 -d z^2$.  
Let $\Omega$ be a domain in ${\mathbb C}$. Let $\zeta= \zeta_1 + i \zeta_2$ be in $\Omega$. Let ${\bf X}(\zeta_1, \zeta_2) = (x(\zeta_1, \zeta_2), y(\zeta_1, \zeta_2), z(\zeta_1, \zeta_2))$ be the Weierstrass-Enneper representation of a maximal surface given by the data $f,g$ as follows:

\begin{equation}
	\begin{split}
		x(\zeta_1, \zeta_2) &= x_0 + \text{Re} \int_{\zeta_0}^{\zeta} ( 1+g^2)(w) f(w)\\
		y(\zeta_1, \zeta_2) &= y_0 + \text{Re} \int_{\zeta_0}^{\zeta} i( 1-g^2)(w) f(w)\\
		z(\zeta_1, \zeta_2) &= z_0 + \text{Re} \int_{\zeta_0}^{\zeta} -2g(w) f(w)
	\end{split}
\end{equation}

Here $f$ is holomorphic and $g$ is meromorphic such that $fg^2$ is holomorphic and $\zeta = \zeta_1 + i \zeta_2$.

Suppose in this neighbourhood, $(\zeta_1, \zeta_2) $ and $(x, y)$ can be transformed into the other, i.e.  there exists an invertible map $\Phi$ such that $(\zeta_1, \zeta_2) = \Phi (x,y)$,

Then $z=z \circ \Phi (x,y)$ is height function  of the maximal surface.  
We name it $z = Z(x,y)$ where $Z(x,y) = z \circ \Phi(x,y)$. It satisfies the equation:

\begin{equation}\label{maxsurf}
(1 -z_x^2) z_{yy} + 2 z_x z_y z_{xy} + (1 - z_y^2) z_{xx} =0.
\end{equation}

\begin{proposition}\label{decompmax}

Let $\Omega \subset {\mathbb C}$ be a domain and $f,g$ be a Weierstrass-Enneper  data on it for a maximal  surface in $ {\mathbb L}^3$ and $\zeta \in \Omega$ be the parameter as above, i.e.$(\zeta_1, \zeta_2) = \Phi (x,y)$ with  $\Phi$ invertible.   Then in the  domain $\Omega$ there is an invertible $\Phi_m$ such that $(\zeta_1, \zeta_2) = \Phi_{m}  (a_m x+b_m, a_m y +d_m)$, when $a_m \neq 0$.   Let $c_m$ be a sequence of non-zero real  numbers.  For  $n \in \{2,3,...\}$,  $z=Z(x,y) = \frac{1}{C_n} \left(\sum_{m=1}^{n } Z_{m} (a_m x+b_m, a_m y+d_m)\right)$ where $z=Z_m(x,y) = \frac{1}{c_m} Z(\frac{x-b_m}{a_m}, \frac{y-d_m}{a_m})$ and $\alpha_m  Z_{m}(x,y)$ is a maximal surface  which is a scaled and translated version of $z=Z(x,y)$ with $\alpha_m = \frac{c_m}{  a_m} $  and $ C_n = \sum_{m=1}^{n}( \frac{1}{c_m}) $ is a constant. 
\end{proposition}

\begin{proof}
The proof is the same as in the minimal surface case.
\end{proof}

\begin{remark}
Decomposition of a family of maximal surfaces can also be exhibited along the same lines as  corollary (~\ref{decompfamily}).
\end{remark}

\subsection{Decomposition of  Timelike Minimal surfaces}

A timelike minimal surface(TLMS)  has a 
 W-E-type representation ~\cite{Le}, ~\cite{Ki}, ~\cite{Ma}:

Let $(u, v) \in \Omega \subset {\mathbb R}^2$ 

${\bf X}= {\bf X}(u,v) =(x(u,v), y(u,v), z(u,v)) $ is a TLMS if

\begin{equation}
	\begin{split}
		x(u,v) &= \frac{1}{2} \int (1+q^2) f(u) du - (1-r^2) g(v) dv\\
		y(u,v) &= - \frac{1}{2} \int (1-q^2) f(u) du + (1-p^2) g(v) dv \\
		z(u,v) &= - \int qf(u) du + r g(v) dv
	\end{split}
\end{equation}

where $(q,r)$ is the projected Gauss map. $f(u)$ and $g(v)$ are arbitrary smooth functions.

Suppose in this domain, there is an invertible  map $\Phi$ such that $\Phi(x,y) = (u,v)$. Then one can write the TLMS as  $z=Z(x,y)$.

\begin{proposition}
Let $\Omega$ be a domain in ${\mathbb R}^2$ and $f=f(u)$ and $g= g(v)$ be the W-E-type data for 
a TLMS    where $(u,v) \in \Omega$.  Suppose there is an invertible transformation $\Phi$ such that $(u,v) = \Phi(x,y)$. Then there exists an invertible $\Phi_m$ such that $(u,v) = \Phi_m(a_m x + b_m, a_m y + d_m)$, where $a_m \neq 0$.  Let us write the TLMS as  $z=Z(x,y)$ locally. Let $c_m$ be a sequence of non-zero real numbers.
Then for every $ n \in \{1,2,3...\}$, $Z(x,y) = \frac{1}{C_n} \sum_{m=1}^{n} Z_{m} (a_m x+ b_m, a_m y + d_m)$ such that $Z_m(x,y) = \frac{1}{c_m} Z(\frac{x-b_m}{a_m}, \frac{y-d_m}{a_m})$. In fact $z = \alpha_m Z_{m}( x, y)  $ is a TLMS which is a scaled and translated  version of  
$z=Z(x,y)$ where $\alpha_m = \frac{c_m}{ a_m} $ and $C_n = \sum_{m=1}^{n} \frac{1}{c_m} $ is a constant.  
\end{proposition}

\begin{proof}

Let $m,n$ be two integers with $m \leq n , m \geq 2$, $a_m \neq 0$. Then

\begin{equation*}
	\begin{split}
	\frac{x}{a_m}(u,v)  &= \frac{1}{2}   \int (1+q^2) \frac{f(u)}{a_m} du - (1-r^2) \frac{g(v)}{a_m} dv    \\
	\frac{y}{a_m}(u,v)  &=  \frac{1}{2}   \int (1+q^2) \frac{f(u)}{a_m}  du - (1-r^2) \frac{g(v)}{a_m} dv     \\
	\frac{z}{a_m}(u,v) &=     - \int q \frac{f(u)}{a_m} du + r \frac{g(v)}{a_m} dv. 
	\end{split}
\end{equation*}

Recall that  solving $z$ in terms of $(x,y)$ from the W-E data, we have $z=z \circ \Phi (x,y)$ is height function  of the minimal surface where $(u,v) = \Phi (x,y)$.  
We name it $z = Z(x,y)$.

Let $z_m(u,v) = \frac{1}{c_m} z(u,v)$.

Existence of $\Phi$ implies there exists an invertible map $\Phi_m $ such that $(u,v) = \Phi_m (a_m x + b_m, a_m y + d_m)$, $a_m \neq 0$. Take $\Phi_m(x,y) = \Phi(\frac{x-b_m}{a_m}, \frac{y-d_m}{a_m})$. 

 Then, $ z_m(u,v) = z_m  \circ \Phi_m (a_m x + b_m, a_m y + d_m)  =     Z_m ( a_m x + b_m, a_m y + d_m) $   where $Z_m = z_m \circ \Phi_m$. Note that $z=\alpha_m Z_m(x,y)$ is a timelike  minimal surface where $\alpha_m = \frac{c_m} {a_m}$.
It is a scaled and translated version of $z=Z(x,y)$.

Then $\sum_{m=1}^{n } z_m (u,v)  =
  \sum_{m=1}^{n }  \frac{1}{c_m}z(u,v ) = z(u,v) (\sum_{m=1}^{n} \frac{1}{c_m} ) =  C_n z(u,v ) $ where $C_n = \sum_{m=1}^{n } \frac{1}{c_m} $.
  
  Thus we have 
  \begin{equation*}\label{minsum}
  z(u,v) = \frac{1}{C_n} \sum_{ m= 1}^{n} z_m(u,v )
  \end{equation*}

  Then $z \circ \Phi (x,y) = \frac{1}{C_n} \sum_{m=1}^{n} z_m \circ \Phi_m ( a_m x + b_m, a_m y + d_m)$.
  
  Let $Z_m = z_m \circ \Phi_m$.
  Thus we have $ Z(x,y) = \frac{1}{C_n} \sum_{m=1}^{n} Z_m( a_m x + b_m, a_m y + c_m )$ with  $Z_m(x,y) = \frac{1}{c_m} Z(\frac{x-b_m}{a_m}, \frac{y-d_m}{a_m})$. Let  $\alpha_m =  \frac{c_m}{a_m}  $. Then   $ z= \alpha_{m}  Z_{m } ( x,y)$ is a timelike minimal surface.  Thus we have the result.

\end{proof}

\subsection{ Decomposition of  Born-Infeld solitons}

1) Barbishov-Charnikov representation

Let us consider the Born-Infeld soliton equation when it is a local graph of the form $z=z(x,y)$, namely 
\begin{equation}\label{(BI)}
(1-z_y^2) z_{xx} + 2 z_x z_y z_{xy} -(1 + z_x^2) z_{yy}=0
\end{equation}

According to Barbishov and Charnikov, the general solution is given by ~\cite{Wi}, page 617:

\begin{equation}
	\begin{split}
		&x = \frac{1}{2} \left( F(r) + G(s)  - \int s^2 G^{\prime}(s) ds -  \int r^2 F^{\prime}(r) dr \right)\\
		&y  =  \frac{1}{2} \left(G(s) -F(r)  - \int r^2 F^{\prime}(r) dr + \int s^2 G^{\prime}(s) ds \right)\\
		&z= \int r F^{\prime}(r) dr + \int s G^{\prime}(s) ds.
	\end{split}
\end{equation}

where $(F, G)$ are two arbitrary smooth functions of  $r$ and $s$ respectively. 

In the domain $\Omega$ where $(r,s)$ is defined, let us assume there exists an invertible transformation $\Phi$ such that $(r,s) = \Phi(x,y)$. 
Then the soliton can be written as $z=Z(x,y)$.

\begin{proposition}\label{sol}
Let $\Omega$ be a domain in ${\mathbb R}^2$ as above and $F=F(r)$ and $G=G(s)$ be a Barbishov-Charnikov data for 
a Born-Infeld soliton where $(r,s) \in \Omega$.  Suppose there is an invertible transformation $\Phi$ such that $(r,s) = \Phi(x,y)$.  Then $z=Z(x,y)$ is a BI-soliton and  $(r,s) = \Phi_{m}  (a_mx+b_m, a_my+d_m)$, $a_m \neq 0$, where $\Phi_{m} $ is invertible. Let $c_m$ be a sequence of non-zero real numbers.    Then for $n \in \{2,3,...\}$, $Z(x,y) = \frac{1}{C_n} \sum_{m=1}^{n} Z_{m} (a_mx+b_m, a_my+d_m)$ where $Z_m(x,y) = \frac{1}{c_m} Z(\frac{x-b_m}{a_m}, \frac{y-d_m}{a_m})$.  Then $z = \alpha_m  Z_{m}(x,y)$ is a BI soliton     which is a scaled and translated version of $z=Z(x,y)$ with $\alpha_m = \frac{c_m}{ a_m} $  and $ C_n = \sum_{m=1}^{n}( \frac{1}{c_m}) $. 
\end{proposition}

\begin{proof}

Let $m \leq n $ with $ m \geq 2$ be two integers. $a_m \neq 0$.
\begin{equation*}
	\begin{split}
		& \frac{x}{a_m}(r,s)  = \frac{1}{2} \left( \frac{F(r)}{a_m} + \frac{G(s)}{a_m}  - \int s^2 \frac{G^{\prime}(s)}{a_m}  ds -  \int r^2 \frac{F^{\prime}(r)}{a_m} dr \right)\\
		& \frac{y}{a_m}(r,s)  =  \frac{1}{2} \left(\frac{G(s)}{a_m} -\frac{F(r)}{a_m}  - \int r^2 \frac{F^{\prime}(r)}{a_m} dr + \int s^2 \frac{G^{\prime}(s)}{a_m} ds \right)\\
		& \frac{z}{a_m}(r,s) = \int r \frac{F^{\prime}}{a_m}(r) dr + \int s \frac{G^{\prime}}{a_m}(s) ds.
	\end{split}
\end{equation*}

Then, $(\frac{x}{a_m},\frac{y}{a_m},\frac{z}{a_m}) $ is a $BI$ solition with Barbishov-Charnikov data $(\frac{F}{a_m}, \frac{G}{a_m})$.

Rest of the proof is exactly same as TLMS (or minimal surface) case with $(u,v)$ (or $(\zeta_1, \zeta_2)$)  replaced by $(r,s)$ wherever it appears.
\end{proof}

\section{{\bf Complex Maximal surface and complex Born-Infeld solitons}}
A solution of the equation (~\ref{maxsurf})
where $(x,y,z)$ are complex will be termed as complex maximal surface. 

If we take an $(x,y,z)$ complex and satisfy the minimal surface equation (~\ref{minsurf}) 
then it is easy to see that the transformation $x \rightarrow ix$ and $y \rightarrow iy$ is a complex maximal surface.

Next we show that the height function of the Scherk's second maximal surface (with complex $(x,y,z)$)  has a finite decomposition.

This can be proved using an E-R identity used in ~\cite{DeSi}. But instead we take a short-cut. Thus we have a proposition as a consequence of   thereom (~\ref{scherk}).

\begin{proposition}\label{scherkmax}
Let $z(x,y) = \ln \left ( \frac{{\rm cosh} y}{{\rm cosh} x} \right) $ be the Scherk's second maximal surface.  Let us consider $(x,y,z)$ to be complex.  Then  $z(x,y) = \sum_{m=0}^{n-1} z_m(x,y)$, where $z_m(x,y)=z \left(\frac{x}{n}+ic(m),\frac{y}{n}+ic(m) \right)$, where $c(m) = \frac{2m-n+1}{2n}\pi $. 
\end{proposition}

\begin{proof}

Let $z = \tilde{z}(x,y) = \ln \left ( \frac{{\rm cos} y}{{\rm cos} x} \right) $ be the Scherk's second minimal surface with  $(x,y,z)$ complex.
Note that  the theorem (~\ref{scherk}) holds with $(x,y,z)$ complex. 
By adapting  the theorem  (~\ref{scherk}) with the transformations $x \rightarrow ix$ and $y \rightarrow iy$
we have 
$\tilde{z}(x,y) =   \sum_{m=0}^{n-1} \tilde{z} \left(\frac{x}{n} - c(m), \frac{y}{n} - c(m) \right)$ where $ c(m) = \frac{(2m-n+1)}{2n}\pi$.  Then we replace $x$, $y$ by $ix$ and $iy$ respectively in $\tilde{z} = \tilde{z} (x,y) $ we get  $\tilde{z}(ix, i y)  = z(x,y)$.  Also, we have $\tilde{z}(x,y) = z(ix,iy)$. Thus we have the following.

\begin{equation*}
	\begin{split}
		z&=z(x,y) = \tilde{z}(ix,iy) = \sum_{m=0}^{n-1} \tilde{z} \left(\frac{ix}{n} - c(m), \frac{iy}{n} - c(m) \right)\\
		  & = \sum_{m=0}^{n-1} z\left( i (\frac{ix}{n} - c(m)), i(\frac{iy}{n} - c(m) )\right) = \sum_{m=0}^{n-1} z \left( - \frac{x}{n} - i c(m) , -\frac{y}{n} - i c(m) \right) \\
		  &   = \sum_{m=0}^{n-1} z \left( \frac{x}{n} + i c(m), \frac{y}{n} + i c(m) \right).
	\end{split}
 \end{equation*} 
\end{proof}

The BI equation (~\ref{(BI)}) (with $(x,y,z)$ complex) can be obtained from the minimal surface equation (with $(x,y,z)$ complex ) complex by a Wick rotation $y \rightarrow iy$. 

$z= \ln \left ( \frac{\text{cosh}(y)}{\text{cos}(x)}\right ) $ with $(x,y,z)$ complex  is an example of a  complex BI soliton surface. 
 
  The following result can be proved using an E-R identity used in ~\cite{De}. But we take a shortcut instead.
 
 We have the following propositon as a consequence of   theorem (~\ref{scherk})
 
 \begin{proposition}\label{scherkBI}
 
 Let $z= \ln \left ( \frac{\text{cosh}(y)}{\text{cos}(x)}\right ) $ be a  complex BI soliton surface, i.e. $(x,y,z)$ are complex. Then, 
$ z(x,y) =\sum_{m=0}^{n-1} z\left( \frac{x}{n}- c(m), \frac{y}{n}+ic(m)\right) $ where $ c(m) = \frac{(2m-n+1)}{2n}\pi$. 
 \end{proposition}

 \begin{proof}
 Let $z = \tilde{z}(x,y) = \ln \left ( \frac{{\rm cos} y}{{\rm cos} x} \right) $ be the Scherk's second minimal surface, with $(x,y,z)$ complex.

By adapting theorem (~\ref{scherk}) to complex $(x,y,z)$  we have 

$\tilde{z}(x,y) =   \sum_{m=0}^{n-1} \tilde{z} \left(\frac{x}{n} - c(m), \frac{y}{n} - c(m) \right)$ where $ c(m) = \frac{(2m-n+1)}{2n}\pi$. 

 If we replace $y$ by $iy$ we get $\tilde{z}(x,iy) = z(x, y)$. We also have $\tilde{z}(x,y) = z(x,iy)$. 
\begin{equation*}
	\begin{split}
		 z &=z(x,y) = \tilde{z}(x,iy) = \sum_{m=0}^{n-1} \tilde{z} \left(\frac{x}{n} - c(m), \frac{iy}{n} - c(m) \right)\\
		  & = \sum_{m=0}^{n-1} z\left(  (\frac{x}{n} - c(m)), i(\frac{iy}{n} - c(m) )\right)= \sum_{m=0}^{n-1} z \left(  \frac{x}{n} -  c(m) , -\frac{y}{n} - i c(m) \right) \\
		  & = \sum_{m=0}^{n-1} z \left( \frac{x}{n} - c(m), \frac{y}{n} + i c(m) \right)
 	\end{split}
 \end{equation*} 
\end{proof}

\section{{\bf Foliation of regions of ${\mathbb E}^3$ using shifted helicoids}}

Let us define a continuous function as follows.

 For each $k \in {\mathbb Z}$,  let $D$ be  the region $(2k-1)\pi \leq x \leq (2k+1)\pi $ minus the points $x = 2k \pi, y=0$. 
 
 On $D$  define $$F(x,y) = (-1)^k Tan^{-1}\frac{y}{x - 2 k \pi}$$ where  $-\frac{\pi}{2} < Tan^{-1}\frac{y}{x} < \frac{\pi}{2}$ be the principal branch of $tan^{-1}\frac{y}{x}$ for  $x,y \in {\mathbb R}$ with $(x,y) \neq (2k \pi,0)$.  
 
 One can check that this is a continuous function and $(x, y, F(x,y)$ is a graph on  $(x,y) \in {\mathbb R}^2 $ minus $(2 k \pi, 0)$, for $k \in {\mathbb Z}$.

Then $(x,y, F(x,y) +t)$ gives a foliation of ${\mathbb E}^3$ minus the  lines $(2 \pi k, 0, z)$, $z \in {\mathbb R}$, $t \in {\mathbb R}$.

 \begin{figure}[!h]
 	\centering
 	\begin{minipage}{.55\textwidth}
 		\centering
 		\includegraphics[width=.5\linewidth]{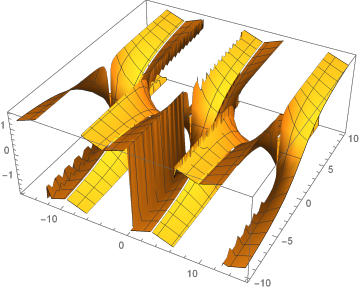}
 		\label{fig:test1}
 	\end{minipage}%
 	\begin{minipage}{.55\textwidth}
 		\centering
 		\includegraphics[width=.5\linewidth]{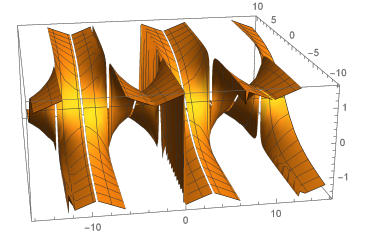}
 		\label{fig:test2}
 	\end{minipage}
 	\captionof{figure}{Foliation by the Shifted Helicoids}
 \end{figure}

 \vskip 1in
 
\section{{\bf Acknowledgement}}

Rukmini Dey would like to acknowledge the support of the DAE, Government of India, under project 12-R$\&$D - TFR-5.10-1100.

\end{document}